\documentclass[12pt]{article}
\usepackage[utf8]{inputenc}
\usepackage[T1]{fontenc}
\usepackage{array}
\usepackage{amsmath, amsthm}
\usepackage{amssymb}
\usepackage{doi}

\newtheorem{theorem}{Theorem}[section]
\newtheorem{proposition}[theorem]{Proposition}
\newtheorem{corollary}[theorem]{Corollary}

\newtheorem{lemma}[theorem]{Lemma}

\newtheorem{remark}[theorem]{Remark}

\DeclareMathOperator{\Suffix}{Suffix} 

\DeclareMathOperator{\rr}{rr} 
\DeclareMathOperator{\Adj}{Adj} 
\DeclareMathOperator{\app}{app} 

\DeclareMathOperator{\toSet}{toSet} 

\DeclareMathOperator{\prb}{p} 

\DeclareMathOperator{\Cand}{A} 

\begin{document}

\title{Secretary problem and two almost the same consecutive applicants}

\author{Josef Rukavicka\thanks{Department of Mathematics,
Faculty of Nuclear Sciences and Physical Engineering, Czech Technical University in Prague
(josef.rukavicka@seznam.cz).}}

\date{\small{June 29, 2021}}

\maketitle

\begin{abstract}
We present a new variant of the secretary problem. Let $\Cand_n$ be a totally ordered set of $n$ \emph{applicants}. Given $P\subseteq \Cand_n$ and $x\in\Cand_n$, let $\rr(P,x)=\vert\{z\in P \mid z\leq x\}\vert\mbox{ }$ be the \emph{relative rank of} $x$ \emph{with regard to} $P$, and let $\rr_n(x)=\rr(\Cand_n,x)$. Let $x_1,x_2,\dots,x_n\in \Cand_n$ be a random sequence of distinct applicants. The aim is to select $1<j\leq n$ such that $\rr_n(x_{j-1})-\rr_n(x_j)\in\{-1,1\}$.

Let $\alpha$ be a real constant with $0<\alpha<1$. Suppose the following stopping rule $\tau_n(\alpha)$: reject first $\alpha n$ applicants and then select the first $x_j$ such that $\rr(P_j,x_{j-1})-\rr(P_j,x_j)\in\{-1,1\}$, where $P_j=\{x_i\mid 1\leq i\leq j\}$. Let $\prb_{n,\tau}(\alpha)$ be the probability that $\tau_n(\alpha)$ selects $x_j$ such that $\rr_n(x_{j-1})-\rr_n(x_j)\in\{-1,1\}$.
We show that \[\lim_{n\rightarrow\infty}\prb_{n,\tau}(\alpha)\leq \lim_{n\rightarrow\infty}\prb_{n,\tau}\left(\frac{1}{2}\right)=\frac{1}{2}\mbox{.}\]

\end{abstract}

\section{Introduction}
The classical \emph{secretary problem} can be formulated as follows:
\begin{itemize}
\item There is a set $\Cand_n$ of $n$ rankable applicants.
\item A company wants to hire the best applicant from the set $\Cand_n$.
\item The applicants come sequentially in a random order to be interviewed by the company. 
\item After interviewing an applicant, the company has to immediately decide if the applicant is selected or rejected.
\item The rejected applicants cannot be recalled.
\item The company knows only the number $n$ and the relative ranks of the applicants being interviewed so far.
\end{itemize}
The first articles presenting a solution to the secretary problem are \cite{Chow1964}, \cite{Dynkin1963}, and \cite{10.2307/2985407}. An optimal stopping rule maximizes the probability of selecting the best applicant. The optimal stopping rule for the classical secretary problem is as follows: Reject the first $ne^{-1}$ applicants and then select the first applicant, who is the best between all the applicants being interviewed so far. The probability of selecting the best applicant with this optimal stopping rule is equal to $e^{-1}$.

Quite many generalizations of the secretary problem have been investigated. To name just a few examples, in \cite{Rose1982} an optimal stopping rule is shown for selecting an applicant, who is ``representative'' for the given set of applicants. The article \cite{Bayon2018} researches an optimal stopping rule for selecting the best or the worst applicant and for selecting the second best applicant. 

Some recent surveys on generalizations of the secretary problem can be found in \cite{Bayon2018}, \cite{10.1214/19-PS333}, and \cite{Szajowski2009}. 

In the current article we consider a variant of the secretary problem, that consists in selecting two consecutive applicants whose relative ranks differ by one. Let $\Cand_n$ be a totally ordered set of $n$ applicants. Given $P\subseteq \Cand_n$ and $x\in\Cand_n$, let $\rr(P,x)=\vert\{z\in P \mid z\leq x\}\vert\mbox{ }$ be the \emph{relative rank of} $x$ \emph{with regard to} $P$, and 
let $\rr_n(x)=\rr(\Cand_n,x)$. We formulate the variant as follows:
\begin{itemize}
\item A company wants to hire two applicants $x,y\in\Cand_n$ such that their relative ranks differ by one; formally $\rr_n(x)-\rr_n(y)\in\{-1,1\}$.
\item The applicants come sequentially in a random order to be interviewed by the company. 
\item After interviewing an applicant, the company has to decide if the previous applicant is selected or rejected. 
\item After interviewing an applicant, the company is allowed to immediately decide if the applicant is selected or rejected.
\item The rejected applicants cannot be recalled.
\item The company has to hire both applicants at the same time immediately after selecting them.
\item The company knows only the number $n$ and the relative ranks of the applicants being interviewed so far.
\end{itemize}
From the formulation it follows that the two selected applicants has to be consecutive in the random sequence.

Let $\alpha$ be a real constant with $0<\alpha<1$ and let $x_1,x_2,\dots,x_n\in\Cand_n$ be a random sequence of distinct applicants. 
Suppose the following stopping rule $\tau_n(\alpha)$: reject first $\alpha n$ applicants and then select the first pair of applicants $x_{j-1}, x_j$ such that $\rr(P_j,x_{j-1})-\rr(P_j,x_j)\in\{-1,1\}$, where \[P_j=\{x_i\mid 1\leq i\leq j\}\mbox{.}\]

Let $\prb_{n,\tau}(\alpha)$ be the probability that $\tau_n(\alpha)$ selects $x_j$ such that $\rr_n(x_{j-1})-\rr_n(x_j)\in\{-1,1\}$.
The main result of the current article is the following theorem.
\begin{theorem}
\label{uwws45124}
If $\alpha$ is a real constant and $0<\alpha<1$ then
\[\lim_{n\rightarrow\infty}\prb_{n,\tau}(\alpha)\leq \lim_{n\rightarrow\infty}\prb_{n,\tau}\left(\frac{1}{2}\right)=\frac{1}{2}\mbox{.}\]
\end{theorem}

Less formally stated, with the given stopping rule $\tau_n(\alpha)$, the optimal strategy is to reject the first half of applicants, and then to select the first consecutive pair of applicants, whose relative ranks differ by one between the applicants being interviewed so far. The probability of success with this rule is equal to $2^{-1}$ as $n$ tends to infinity.


\begin{remark}
The sequence A002464 (\url{https://oeis.org/A002464}) expresses ``the number of permutations of length n without rising or falling successions''. 
Let $a(n)$ denote this integer sequence A002464. It is known \cite{BAGNO2020103119,doi:10.1080/00029890.1980.11994973}, that $\lim_{n\rightarrow\infty}\frac{a(n)}{n!}=e^{-2}$.

It follows that if $x_1,x_2,\dots,x_n\in\Cand_n$ is a random sequence of $n$ distinct applicants, then with probability $e^{-2}$ there is no $j\in\{2,\dots,n\}$ such that $\rr_n(x_j)-\rr_n(x_{j-1})\in\{-1,1\}$  as $n$ tends to infinity. 
\end{remark}

\section{Preliminaries}
For the whole article, suppose that $n>3$, where $n=\vert \Cand_n\vert$.

Let $\mathbb{R}^+$ denote the set of all positive real numbers, let $\mathbb{Z}$ denote the set of all integers, and let $\mathbb{N}^+$ denote the set of all positive integers.


Given $j\in\mathbb{N}^+$, let $\Cand_n^{j}=\{(x_1,\dots,x_j)\mid x_i\in\Cand_n\mbox{ for all }i\in\{1,2,\dots,j\}\}$ and let $\Cand_n^{+}=\bigcup_{j\in\mathbb{N}^+}\Cand_n^j$. The elements of $\Cand_n^+$ are called \emph{sequences of applicants} or just $\emph{sequences}$.

Suppose $\vec{x},\vec{y}\in\Cand_n^+$, where $\vec{x}=(x_1,x_2,\dots,x_i)$ and $\vec{y}=(y_1,y_2,\dots,y_j)$. Let $\vec{x}\circ\vec{y}\in\Cand_n^{i+j}$ denote the concatenation of $\vec{x}$ and $\vec{y}$; formally \[\vec{x}\circ\vec{y}=(x_1,x_2,\dots,x_i,y_1,y_2\dots,y_j)\in\Cand_n^{i+j}\mbox{.}\]
\begin{remark}
We consider that $\Cand_n=\Cand_n^1$; it means that if $x\in \Cand_n$ then $x=(x)$ is a sequence of length $1$.
\end{remark}

Given $\vec{x}=(x_1,x_2,\dots,x_k)\in\Cand_n^+$, $i,j\in\{1,2,\dots,k\}$, and $i\leq j$, let $\vec{x}[i]=x_i$, let $\vec{x}[i,j]=(x_i,x_{i+1}, \dots, x_j)\in\Cand_n^{j-i+1}$,  let $\vert \vec{x}\vert=k$ denote the length of $\vec{x}$, and let $\toSet(\vec{x})=\{\vec{x}[i]\mid 1\leq i\leq \vert \vec{x}\vert\}$ be the set of applicants of the sequence $\vec{x}$. 

Given $k\in\mathbb{N}^+$, $k\leq n$, let \[\begin{split}\Omega_n(k)=\{\vec{x}\in\Cand_n^k\mid \vec{x}[i]=\vec{x}[j]\mbox{ implies }i=j \mbox{ for all }i,j\in\{1,2,\dots,k\}\}\mbox{.}\end{split}\]
Let $\Omega_n=\Omega_n(n)$.
Obviously $\vert \Omega_n(k)\vert=\frac{n!}{(n-k)!}$ and in particular $\vert \Omega_n\vert=n!$. 
\begin{remark}
The set $\Omega_n(k)$ contains sequences of $k$ distinct applicants. The sequences of $\Omega_n$ represent the sequences of applicants that the company is interviewing when selecting the applicants.
\end{remark}

Given $k\in\mathbb{N}^+$, let $\Delta(k)=\{-1,1,k -1, 1-k\}\subseteq\mathbb{Z}$. 
Given $P\subseteq \Cand_n$,  
let \[\begin{split}\Adj(P)=\{\{x,y\}\mid x,y\in P\mbox{ and } \rr(P,x)-\rr(P,y)\in \Delta(\vert P\vert)\}\mbox{.}\end{split}\]
The set $\Adj(P)$ contains all sets $\{x,y\}$ such that the difference of relative ranks of $x,y$ with regard to $P$ is from the set $\Delta(\vert P\vert)$. We say that the applicants $x,y$ are \emph{adjacent} with regard to $P$ if $\{x,y\}\in\Adj(P)$.

\begin{remark}
The set $\Delta(k)$ contains the differences in relative ranks, that ``the company is looking for''; it means $\{-1,1\}\subseteq\Delta(k)$. In addition the set $\Delta(k)$ contains values $1-k,k-1$. In consequence the applicants $\max\{P\}$ and $\min\{P\}$ are adjacent with regard to $P$.
\end{remark}

\section{Secretary problem}

Given $r,k\in\{3,4,\dots,n-1\}$ and $r\leq k$, let \[\begin{split}\Lambda_n(r,k)=\{\vec{x}\in\Omega_n\mid \{\vec{x}[k],\vec{x}[k+1]\}\in\Adj(\Cand_n) \mbox{ and }\\ \{\vec{x}[i-1],\vec{x}[i]\}\not \in\Adj(\toSet(\vec{x}[1,i])\}\mbox{ for all }i\in\{r,r+1,\dots,k\}\}\mbox{.}\end{split}\]
Let $\Lambda_n(r)=\bigcup_{k=r}^{n-1}\Lambda_n(r,k)\mbox{.}$ Obviously $\Lambda_n(r,k)\cap \Lambda_n(r,\overline k)=\emptyset$ if $k\not=\overline k$.

\begin{remark}
\label{thf87ejdh}
Let $\widetilde \tau_n(r)$ denote the following stopping rule: Given $\vec{x}\in\Omega_n$, reject first $r-1$ applicants and then select the first pair $\vec{x}[k],\vec{x}[k+1]$ such that $\{\vec{x}[k],\vec{x}[k+1]\}\in\Adj(\toSet(\vec{x}[1,k+1]))$.

It is clear that given a sequence $\vec{x}\in\Lambda_n(r,k)$, the stopping rule $\widetilde\tau_n(r)$ would select the applicants $\vec{x}[k], \vec{x}[k+1]$. Also, on the other hand, if $\vec{x}\in\Omega_n$ and the stopping rule $\widetilde\tau_n(r)$ selects the applicants $\vec{x}[k], \vec{x}[k+1]$ then $\vec{x}\in\Lambda_n(r,k)$. 
\end{remark}

Given $r\in\{3,4,\dots,n-1\}$, let \[\begin{split}\Pi_n(r)=\{\vec{x}\in\Lambda_n(r)\mid \rr_n(\vec{x}[k])-\rr_n(\vec{x}[k+1])\in\{1-n,n-1\}\mbox{, }\\ \mbox{ where }k\mbox{ is such that }\vec{x}\in\Lambda_n(r,k)\}\mbox{.}\end{split}\]

\begin{remark}
Note that if $\vec{x}\in\Lambda_n(r)$ then there is exactly one $k\in\{r,r+1, \dots,n-1\}$ such that $\vec{x}\in\Lambda_n(r,k)$. Thus the definition of $\Pi_n(r)$ makes sense.
\end{remark}

We express the probability $\prb_{n,\tau}(\alpha)$ by means of sizes of sets.
\begin{lemma}
\label{rud883kfjh}
If $\alpha\in\mathbb{R}^+$, $\alpha<1$, and $\lfloor\alpha n\rfloor\geq 3$ then 
\[\prb_{n,\tau}(\alpha)=\frac{\vert \Lambda_n(\lfloor\alpha n\rfloor)\setminus\Pi_n(\lfloor\alpha n\rfloor)\vert}{\vert \Omega_n\vert}\mbox{.}\]
\end{lemma}
\begin{proof}
Let $r=\lfloor\alpha n\rfloor$ and let $k\in\{r,r+1,\dots,n-1\}$.
Given $\vec{x}\in\Lambda_n(r)$, obviously we have that \[\begin{split}\vec{x}\in\left(\Lambda_n(r)\setminus\Pi_n(r)\right)\cap \Lambda_n(r,k)\mbox{ if an only if}\\ \rr_n(\vec{x}[k])-\rr_n(\vec{x}[k+1])\in\{-1,1\}\mbox{.}\end{split}\]
The lemma follows then from the definition of $\Lambda_n(r)$ and Remark \ref{thf87ejdh}. This completes the proof.
\end{proof}

We present a technical lemma, that we will apply in the proof of Proposition \ref{eyr63yt7}.
\begin{lemma}
\label{ey77f983jh}
If $P\subseteq \Cand_n$, $\vert P\vert\geq 3$, $x\in P$, and 
\[\upsilon_n(P,x)=\vert\{y\in P\mid y\not=x\mbox{ and }\{x,y\}\in\Adj(P)\}\vert\mbox{, }\]
then $\upsilon_n(P,x)=2$.
\end{lemma}
\begin{proof}
Clearly we have that
\begin{itemize}
\item 
If $x\in P\setminus\{\max\{P\}\}$ then there is exactly exactly one applicant $y\in P$ such that $\rr(P,x)-\rr(P,y)=-1$.
\item 
If $x\in P\setminus\{\min\{P\}\}$ then there is exactly exactly one applicant $y\in P$ such that $\rr(P,x)-\rr(P,y)=1$.
\item
If $\vert P\vert=k$ then 
\[\rr(P,\min\{P\})-\rr(P,\max\{P\})=1-k\in\Delta(k)\mbox{}\]
and \[\rr(P,\max\{P\})-\rr(P,\min\{P\})=k-1\in\Delta(k)\mbox{.}\]
\end{itemize}
Since $\Delta(k)=\{-1,1,k -1, 1-k\}$ the lemma follows.
\end{proof}
\begin{remark}
Lemma \ref{ey77f983jh} is the reason that we have defined that $\max\{P\}$ and $\min\{P\}$ are adjacent with regard to $P$. Otherwise the value of $\upsilon_n(P,x)$ would depend on $x\in P$. As a result the proof of the next proposition will be more simple.
\end{remark}

Given $r,k\in\{3,4,\dots,n-1\}$, $r\leq k$, and $z,\overline z\in\Cand_n$, let \[\Lambda(r,k,z,\overline z)=\{\vec{x}\in\Lambda(r,k)\mid \vec{x}
[k]=z\mbox{ and }\vec{x}[k+1]=\overline z\}\mbox{.}\]

The sets $\Lambda(r,k,z,\overline z)$ form a partition of the set $\Lambda(r,k)$. We derive a formula for the size of the sets $\Lambda(r,k,z,\overline z)$.
\begin{proposition}
\label{eyr63yt7}
If $r,k\in\{3,4,\dots,n-1\}$, $r\leq k$, $z,\overline z\in\Cand_n$, and $\{z,\overline z\}\in\Adj(\Cand_n)$ then \[\vert \Lambda_n(r,k,z,\overline z)\vert=(n-2)!\frac{(r-2)(r-3)}{(k-1)(k-2)}\mbox{.}\]
\end{proposition}
\begin{proof}
Given $\vec{y}\in \Cand_n^+$, let \[\omega(\vec{y})=\Cand_n\setminus\left(\{z,\overline z\}\cup \toSet(\vec{y})\right)\mbox{.}\]

Given $j\in\mathbb{N}^+$ and $D\subseteq \Cand_n^{+}$, let \[\begin{split}\Suffix_n(D,j)=\{\vec{y}\in\Cand_n^j\mid \mbox{ there is }\vec{x}\in\Cand_n^+ \mbox{ such that }\vec{x}\circ\vec{y}\in D\}\mbox{.}\end{split}\]

For $j\in\{1,2,\dots,n\}$, we define the sets $H(j)\subseteq \Suffix_n(\Omega_n,n-j+1)$ as follows.
\begin{itemize}
\item
Let 
\begin{equation*}
    H(n) =
    \begin{cases}
      \{\overline z\} & \mbox{ if } k+1=n \\
      \{x\mid x\in\Cand_n\setminus\{z,\overline z\}\}        & \mbox{ otherwise.}
    \end{cases}
  \end{equation*}
\item
Given $j\in\{k+2,k+2,\dots,n-1\}$, let 
\[\begin{split}H(j)=\{(x)\circ\vec{y}\mid \vec{y}\in H(j+1) \mbox{ and } x\in\omega(\vec{y})\}\mbox{.}\end{split}\]
\item
Let 
\begin{equation*}
    H(k+1) =
    \begin{cases}
      H(n)=\{\overline z\} & \mbox{ if } k+1=n \\
      \{(\overline z)\circ\vec{y}\mid \vec{y}\in H(k+2)\}        & \mbox{ otherwise.}
    \end{cases}
  \end{equation*}
\item
Let $H(k)=\{(z)\circ\vec{y}\mid \vec{y}\in H(k+1)\}$.
\item
Given $j\in\{r-1,r,\dots,k-1\}$, let 
\[\begin{split}H(j)=\{(x)\circ\vec{y} \mid \vec{y}\in H(j+1)\mbox{ and }x\in\omega(\vec{y})\mbox{ and }\\ \{x,\vec{y}[1]\}\not\in\Adj(\omega(\vec{y})\cup\{\vec{y}[1]\})\}\mbox{.}\end{split}\]
\item
Given $j\in\{1,2,\dots,r-2\}$, let 
\[\begin{split}H(j)=\{(x)\circ\vec{y}\mid \vec{y}\in H(j+1) \mbox{ and } x\in\omega(\vec{y})\}\mbox{.}\end{split}\]
\end{itemize}

It is straightforward to verify that $\Lambda(r,k,z,\overline z)=H(1)$.
We derive the formulas for the size of $H(j)$. Note that if $\vec{y}\in H(j+1)\subseteq \Suffix_n(\Omega_n,n-j)$, then $\vert\vec{y}\vert=n-j$. From the definition of $H(j)$ it follows that 
\begin{itemize}
\item
\begin{equation*}
    \vert H(n)\vert =
    \begin{cases}
      1 & \mbox{ if } k+1=n \\
      n-2        & \mbox{ otherwise.}
    \end{cases}
  \end{equation*}
\item
If $j\in\{k+2,k+2,\dots,n-1\}$ then
$\vert H(j)\vert=(j-2)\vert H(j+1)\vert$. 

Realize that if $\vec{y}\in H(j+1)$ then $\vert \vec{y}\vert=n-j$, $\vert\{z,\overline z\}\vert=2$, and $\toSet(\vec{y})\cap\{z,\overline z\}=\emptyset$. It follows that $\vert\omega(\vec{y})\vert=(n-(n-j)-2)=j-2$.
\item
\begin{equation*}
    \vert H(k+1)\vert =
    \begin{cases}
      \vert \{\overline z\}\vert=1 & \mbox{ if } k+1=n \\
      \vert H(k+2)\vert        & \mbox{ otherwise.}
    \end{cases}
  \end{equation*}
\item
$\vert H(k)\vert=\vert H(k+1)\vert$.
\item
If $j\in\{r-1,r,\dots,k-1\}$ then
$\vert H(j)\vert=(j-2)\vert H(j+1)\vert$. 

Realize that if $\vec{y}\in H(j+1)$ then $\vert \vec{y}\vert=n-j$ and $\{z,\overline z\}\subseteq \toSet(\vec{y})$. It follows that $\vert\omega(\vec{y})\vert=(n-(n-j))=j$. Moreover realize that there are exactly two distinct applicants $x,\overline x\in\omega(\vec{y})$ such that $\{x,\vec{y}[1]\}, \{\overline x,\vec{y}[1]\}\in\Adj(\omega(\vec{y})\cup\{\vec{y}[1]\})$; see Lemma \ref{ey77f983jh}.

\item
If $j\in\{1,2,\dots,r-2\}$ then
$\vert H(j)\vert=(n-(n-j))\vert H(j+1)\vert=j\vert H(j+1)\vert$. 

Realize that if $\vec{y}\in H(j+1)$ then $\vert \vec{y}\vert=n-j$ and $\{z,\overline z\}\subseteq \toSet(\vec{y})$. It follows that $\vert\omega(\vec{y})\vert=(n-(n-j))=j$.
\end{itemize}

Let $g(j)=\frac{\vert H(j)\vert}{\vert H(j+1)\vert}$ for $j\in\{1,2,\dots,n-1\}$. The next table shows the values of $g(j)$.
\begin{table}[h]
\centering
\begin{tabular}{|l|l|l|l|l|l|l|l|l|l|l|l|l|}
\hline
\textbf{j}  & n & n-1 & \dots & k+2 & k+1 & k & k-1 &\dots &  r & r-1 & r-2 & \dots \\ \hline
\textbf{g(j)}  & n-2 & n-3 & \dots & k & 1 & 1 & k-3 &\dots &  r-2 & r-3 & r-2 & \dots \\ \hline
\end{tabular}
\end{table}
From the table we can derive that \[\vert H(1)\vert=(n-2)!\frac{(r-2)(r-3)}{(k-1)(k-2)}\mbox{.}\]
Since $\Lambda(r,k,z,\overline z)=H(1)$ this completes the proof.
\end{proof}
\begin{remark}
From Proposition \ref{eyr63yt7} it follows that $\Lambda_n(3,k,z,\overline z)=\emptyset$. This is correct, since for every $\vec{x}\in\Omega_n$, we have that $\{\vec{x}[2],\vec{x}[3]\}\in\Adj(P)$, where $P=\{\vec{x}[1], \vec{x}[2], \vec{x}[3]\}$.
\end{remark}

Using Proposition \ref{eyr63yt7}, we can present a formula for the size of $\Lambda_n(r,k)$.
\begin{lemma}
\label{dyd6e7hgfuye9}
If $r,k\in\{3,4,\dots,n-1\}$, and $r\leq k$ then \[\vert \Lambda_n(r,k)\vert=2n(n-2)!\frac{(r-2)(r-3)}{(k-1)(k-2)}\mbox{.}\]
\end{lemma}
\begin{proof}
It is clear that 
\begin{itemize}\item$\Lambda_n(r,k)=\bigcup_{(z,\overline z)\in\Cand_n^2}\Lambda_n(r,k,z,\overline z)$ and \item if $z_1,z_2,z_3,z_4\in\Cand_n$ and $(z_1,z_2)\not=(z_3,z_4)$ then \[\Lambda_n(r,k,z_1,z_2)\cap\Lambda_n(r,k,z_3,z_4)=\emptyset\mbox{.}\]
\item If $z,\overline z\in \Cand_n$ and $\{z,\overline z\}\not \in \Adj(\Cand_n)$ then from the definition of $\Lambda_n(k,r)$ and $\Lambda_n(k,r,z,\overline z)$ we have that $\Lambda_n(r,k,z,\overline z)=\emptyset$.
\end{itemize}
Let $T=\{(z,\overline z)\in\Cand_n^2\mid \{z,\overline z\}\in\Adj(\Cand_n)\}$. It follows then that 
\begin{equation}\label{hg6cvsb2}\begin{split}\vert \Lambda_n(r,k)\vert=\sum_{(z,\overline z)\in T}\vert \Lambda_n(r,k,z,\overline z)\vert\mbox{.}\end{split}\end{equation}
 Obviously we have that \begin{equation}\label{tyf764g}\vert T\vert=2n\mbox{.}\end{equation} The lemma follows from (\ref{hg6cvsb2}), (\ref{tyf764g}), and Proposition \ref{eyr63yt7}. This completes the proof.
\end{proof}

Given a sequence $\vec{x}\in\Omega_n$ and $r\in\{3,4,\dots,n-1\}$, the next theorem shows the probability that $\vec{x}\in\Lambda_n(r)$.
\begin{theorem}
\label{dieief887e}
If $r\in\{3,4,\dots,n-1\}$ then \[\frac{\vert \Lambda_n(r)\vert}{\vert\Omega_n\vert}=2\frac{r-3}{n-1}-2\frac{(r-2)(r-3)}{(n-1)(n-2)}\mbox{.}\]
\end{theorem}
\begin{proof}
Since $\Lambda_n(r)=\bigcup_{k\geq r}^{n-1}\Lambda_n(r,k)$ and $\Lambda_n(r,k)\cap\Lambda_n(r,\overline k)=\emptyset$ if $k\not=\overline k$, 
we have that \begin{equation}\label{ghry538544}\vert\Lambda_n(r)\vert=\sum_{k=r}^{n-1}\vert\Lambda_n(r,k)\vert\mbox{.}\end{equation}

Recall that $\vert \Omega_n\vert=n!$. Then from Lemma \ref{dyd6e7hgfuye9} and (\ref{ghry538544}) it follows that 
\begin{equation}\label{vhf76udh}\begin{split}\frac{\vert \Lambda_n(r)\vert}{\vert\Omega_n\vert}=\sum_{k=r}^{n-1} \frac{\vert \Lambda_n(r,k)\vert}{n!}=\sum_{k=r}^{n-1}2n\frac{(n-2)!}{n!}\frac{(r-2)(r-3)}{(k-1)(k-2)}= \\ 2\frac{(r-2)(r-3)}{n-1}\sum_{k=r}^{n-1}\frac{1}{(k-1)(k-2)}\mbox{.}\end{split}\end{equation}

We have that 
\begin{equation}\label{hgeyr67}\begin{split}\sum_{k=r}^{n-1}\frac{1}{(k-1)(k-2)}= \sum_{k=r}^{n-1}\left(\frac{1}{k-2}-\frac{1}{k-1}\right)=\\ \frac{1}{r-2}-\frac{1}{r-1} + \frac{1}{r-1}-\frac{1}{r}\dots \frac{1}{n-4}-\frac{1}{n-3} +\frac{1}{n-3}-\frac{1}{n-2}=\\ \frac{1}{r-2}-\frac{1}{n-2}\mbox{.}\end{split}\end{equation}

From (\ref{vhf76udh}) and (\ref{hgeyr67}) it follows that
\begin{equation}\label{b628djll9}\begin{split}\frac{\vert \Lambda_n(r)\vert}{\vert\Omega_n\vert}=2\frac{(r-2)(r-3)}{n-1}\left(\frac{1}{r-2}-\frac{1}{n-2}\right) = \\ 2\frac{r-3}{n-1}-2\frac{(r-2)(r-3)}{(n-1)(n-2)}\mbox{.}\end{split}\end{equation}

This completes the proof.
\end{proof}

Given a sequence $\vec{x}\in\Omega_n$, $\alpha\in\mathbb{R}^+$, and $\alpha<1$, the next lemma shows the probability that $\vec{x}\in\Lambda_n(\alpha n)$ as $n$ tends to infinity
\begin{lemma}
\label{eye6r8u}
If $\alpha\in\mathbb{R}^+$, $\alpha<1$ then 
\[\lim_{n\rightarrow \infty}\frac{\vert \Lambda_n(\lfloor\alpha n\rfloor)\vert}{\vert\Omega_n\vert}\leq \lim_{n\rightarrow \infty}\frac{\vert \Lambda_n(\lfloor\frac{1}{2} n\rfloor)\vert}{\vert\Omega_n\vert}=\frac{1}{2}\mbox{.}\]
\end{lemma}
\begin{proof}
From Theorem \ref{dieief887e} we get that
\begin{equation}\label{yt7yyu3bv}\begin{split}\lim_{n\rightarrow\infty}\frac{\vert \Lambda_n(\lfloor\alpha n\rfloor)\vert}{\vert\Omega_n\vert}=\lim_{n\rightarrow\infty}\left(2\frac{\alpha n-1}{n-1}-2\frac{(\alpha n-2)(\alpha n-3)}{(n-1)(n-2)}\right)=2\alpha-2\alpha^2\mbox{.}\end{split}\end{equation}

It is easy to verify that the function $f(\alpha)=2\alpha-2\alpha^2$ has a maximum for $\alpha=2^{-1}$ and that $f(2^{-1})=2^{-1}$.  This ends the proof.
\end{proof}

We show the relation between the size of $\Pi_n(r)$ and $\Lambda_n(r)$.
\begin{lemma}
\label{oope0rt9jj}
If $r\in\{3,4,\dots,n-1\}$ then \[\frac{\vert \Pi_n(r)\vert}{\vert \Lambda_n(r)\vert}=\frac{1}{n}\mbox{.}\]
\end{lemma}
\begin{proof}
Given $m\in\mathbb{Z}$, let $\app(m)=x\in\Cand_n$, where $x$ is such that \[\rr_n(x)=1+(m\bmod n)\mbox{.}\]
Less formally said, the function $\app(m)$ returns the applicant $x\in\Cand_n$ with the relative rank $m$ taken modulo $n$.

Given $x\in \Cand_n$, let $\sigma(x)=\app(\rr_n(x)+1)\in\Cand_n$. Less formally said, the function $\sigma(x)$ returns the applicant $y$ such that $y$ has a higher relative rank and $x,y$ are adjacent, where $x\not=\max\{\Cand_n\}$. If $x=\max\{\Cand_n\}$, then $\sigma(x)=\min\{\Cand_n\}$.

We define that $\sigma(x)=\sigma^1(x)$ and that $\sigma^{i+1}(x)=\sigma(\sigma^i(x))$ for all $i\in \mathbb{N}^+$.
It is clear that $\sigma(x):\Cand_n\rightarrow\Cand_n$ is a bijection and that $\sigma^n(x)=x$. 

If $\vec{x}=(x_1,x_2,\dots,x_n)\in\Omega_n$ and $k\in\mathbb{N}$ then let \[\sigma^k(\vec{x})=(\sigma^k(x_1), \sigma^k(x_2), \dots,\sigma^k(x_n))\mbox{.}\]
Obviously $\sigma^k:\Omega_n\rightarrow\Omega_n$ is a bijection, and $\sigma^n(\vec{x})=\vec{x}$.

Then it is straightforward to verify that if $\vec{x}\in\Pi_n(r)$  and \[G=\{\sigma^i(\vec{x})\mid i\in\{1,2,\dots,n-1\}\}\] then $\vert G\vert=n-1$ and $G\subseteq\Lambda_n(r)\setminus\Pi_n(r)$.
Since $\sigma^n(\vec{x})=\vec{x}$, the lemma follows.
\end{proof}

From Lemma \ref{oope0rt9jj} we have the following obvious result.
\begin{lemma}
\label{cmalfj30}
If $\alpha\in\mathbb{R}^+$, $\alpha<1$ then 
\[\lim_{n\rightarrow \infty}\frac{\vert \Lambda_n(\lfloor\alpha n\rfloor)\setminus\Pi_n(\lfloor\alpha n\rfloor)\vert}{\vert \Lambda_n(\lfloor\alpha n\rfloor)\vert}=1\mbox{.}\]
\end{lemma}

From Lemma \ref{rud883kfjh} and Lemma \ref{cmalfj30} we have the next corollary.
\begin{corollary}
\label{ur7fue99ufhn}
If $\alpha\in\mathbb{R}^+$, $\alpha<1$ then \[\lim_{n\rightarrow\infty}\prb_{n,\tau}(\alpha)=\lim_{n\rightarrow\infty} \frac{\vert \Lambda_n(\lfloor\alpha n\rfloor)\vert}{\vert \Omega_n\vert}\mbox{.}\]
\end{corollary}

Then Theorem \ref{uwws45124} follows from  Lemma \ref{eye6r8u} and Corollary \ref{ur7fue99ufhn}.

\section*{Acknowledgments}
This work was supported by the Grant Agency of the Czech Technical University in Prague, grant No. SGS20/183/OHK4/3T/14.

\bibliographystyle{siam}
\IfFileExists{biblio.bib}{\bibliography{biblio}}{\bibliography{../!bibliography/biblio}}

\end{document}